\documentclass[12pt]{amsart}

\newtheorem{theorem}{Theorem}[section]
\newtheorem{lemma}[theorem]{Lemma}
\newtheorem{corollary}[theorem]{Corollary}
\newtheorem{proposition}[theorem]{Proposition}
\newtheorem{remark}[theorem]{Remark}

\newcommand{\ncom}{\newcommand}
\ncom{\ep}{\epsilon}
\ncom{\rar}{\rightarrow}
\ncom{\thrar}{\twoheadrightarrow}
\ncom{\lrar}{\longrightarrow}
\ncom{\ov}{\overline}
\ncom{\what}{\widehat}

\newcommand{\ignore}[1]{}
\ncom{\m}{\mbox}
\ncom{\sta}{\stackrel}
\ncom{\C}{{\mathbb C}}
\ncom{\A}{{\mathbb A}}
\ncom{\Z}{{\mathbb Z}}
\ncom{\Q}{{\mathbb Q}}
\ncom{\R}{{\mathbb R}}
\ncom{\G}{{\mathbb G}}
\ncom{\HH}{{\mathbb H}}
\ncom{\al}{\alpha}
\ncom{\p}{{\mathbb P}}
\ncom{\N}{{\mathbb N}}
\ncom{\K}{{\mathbb K}}
\ncom{\X}{{\mathbb X}}
\ncom{\f}{\frac}
\ncom{\cA}{{\mathcal A}}
\ncom{\cB}{{\mathcal B}}
\ncom{\cD}{{\mathcal D}}
\ncom{\cDB}{{\mathcal D \mathcal B}}
\ncom{\cX}{{\mathcal X}}
\ncom{\cO}{{\mathcal O}}
\ncom{\cW}{{\mathcal W}}
\ncom{\cL}{{\mathcal L}}
\ncom{\cP}{{\mathcal P}}
\ncom{\cH}{{\mathcal H}}
\ncom{\cS}{{\mathcal S}}
\ncom{\cM}{{\mathcal M}}
\ncom{\cC}{{\mathcal C}}
\ncom{\cK}{{\mathcal K}}
\ncom{\cT}{{\mathcal T}}
\ncom{\cF}{{\mathcal F}}
\ncom{\cN}{{\mathcal N}}
\ncom{\cJ}{{\mathcal J}}
\ncom{\cV}{{\mathcal V}}
\ncom{\cZ}{{\mathcal Z}}
\ncom{\cU}{{\mathcal U}}
\ncom{\cSU}{{\mathcal S \mathcal U}}
\ncom{\cG}{{\mathcal G}}
\ncom{\cQ}{{\mathcal Q}}
\ncom{\cR}{{\mathcal R}}
\ncom{\cY}{{\mathcal Y}}
\ncom{\cE}{{\mathcal E}}
\ncom{\cI}{{\mathcal I}}
\ncom{\mylabel}[1]{{\rm (#1)}\label{#1}}
\ncom{\Hom}{{\textit{Hom}}}
\ncom{\eop}{{\hfill $\Box$}}
\begin{document}
\baselineskip=16pt


\title[Stability of tangent bundle]{Stability of tangent bundle on the moduli space of stable bundles on a curve }


\author[J. N. Iyer]{Jaya NN  Iyer}

\address{The Institute of Mathematical Sciences, CIT
Campus, Taramani, Chennai 600113, India}
\email{jniyer@imsc.res.in}

\footnotetext{Mathematics Classification Number: 53C55, 53C07, 53C29, 53.50. }

\begin{abstract}
In this paper, we prove that the tangent bundle of the moduli space $\cSU_C(r,d)$ of stable bundles of rank $r$ and of fixed determinant of degree $d$ (such that $(r,d)=1$), on a smooth projective curve $C$ is always stable, in the sense of Mumford-Takemoto. This proves a conjecture and is related to a conjectural existence of a K\"ahler-Einstein metric on Fano varieties with Picard number one.
\end{abstract}
\maketitle

\setcounter{tocdepth}{1}
\tableofcontents

\section{Introduction}

Suppose $X$ is a compact K\"ahler manifold. The existence of a K\"ahler-Einstein metric on $X$ has attracted wide interest since the conjecture of Calabi and work of Yau \cite{Yau} appeared, in the study of complex manifolds.  Aubin \cite{Aubin} and Yau  show the existence of a K\"ahler-Einstein metric whenever the canonical line bundle $K_X$ is ample or trivial. The existence of a K\"ahler-Einstein metric when $-K_X$ is ample, i.e., when $X$ is a Fano manifold, is an open problem.
This has many interesting applications and are discussed by Tian in \cite{Tian2}.
 Kobayashi \cite{Kobayashi} and L\"ubke \cite{Lubke}  show that the existence of a K\"ahler-Einstein metric implies the stability of the tangent bundle, in the sense of Mumford and Takemoto.
In particular, the tangent bundle $T_X$ is stable when  $X$ is of general type. Since then the stability problem for Fano manifolds has brought a lot of attention. A very recent announcement on $K$-stability and existence of K\"ahler-Einstein metrics is made in \cite{Chen}, by Chen-Donaldson-Sun, and by 
G. Tian \cite{Tianrecent}.

The significant works of  Hwang \cite{Hwang}, Peternell-Wisniewski \cite{Peternell}, Steffens \cite{Steffens},  Subramanian \cite{Subramanian}, Tian \cite{Tian} (and the references therein), prove the stability result for certain Fano manifolds $X$. In most of these cases, the Betti number $b_2(X)=1$.
When  $b_2(X)>1$, some examples are known when the stability fails, see \cite[p.183]{Tian2}.  Since then it was speculated (for instance, by Peternell \cite[p.14, Conjecture 5.2]{Peternell2})
 that the stability of the bundle $T_X$ holds when $X$ is a Fano manifold with $b_2(X)=1$. The list of examples where this is known to hold is rather small, and we investigate this problem for the following important class of varieties, namely the moduli spaces of stable bundles on a curve.

Suppose $C$ is a smooth projective curve of genus $g$. The moduli space $\cSU_C(r,d)$ of stable vector bundles of rank $r$ and of fixed determinant of degree $d$, is a projective Fano manifold when $r$ and $d$ are coprime.
Furthermore, the Picard number is one, generated by the \textit{determinant line bundle} $L$ and the canonical class is $K=L^{-2}$. In other words, the moduli space is of index $2$ \cite{Ramanan}. When the rank $r=2$ and $g\geq 2$, Hwang \cite[Theorem 1]{Hwang2}, proved the stability of the tangent bundle on $\cSU_C(2,1)$.

In this paper, we prove the stability of the tangent bundle, for higher rank smooth projective moduli spaces.
More precisely,

\begin{theorem}\label{MainTheorem}
Suppose $r\geq 3$ and $d$ is an integer such that $(r,d)=1$. Suppose $C$ is a smooth projective curve of genus $g(C)\geq 3$. Then the tangent bundle on the moduli space $\cSU_C(r,d)$ is always stable, in the sense of Mumford-Takemoto.
\end{theorem}


We use the Hecke correspondence \cite{Narasimhan},\cite[p.206]{Beauville}, relating the moduli spaces $\cSU_C(r,1)$ and $\cSU_C(r,1-h)$ for any $h,0<h<r$. In fact, it gives
 a correspondence given by a Grassmannian bundle, as shown in \cite{Beauville}. The key point is to use the structure of this correspondence and prove that the stability of the tangent bundle of the respective moduli spaces is preserved under this correspondence (see Proposition \ref{heckestable}). 
The proof is now by assuming to the contrary and consider destabilizing subsheaves of the cotangent sheaves. 
We consider the sum of the sheaves inside the direct sum of the relative cotangent sheaves on the Grassmannian bundle. The kernel subsheaf of the direct sum and the image are investigated,
 see \eqref{exeexactU}. A careful analysis of induced sections in the Hodge cohomologies twisted by appropriate powers of ample generators of Picard groups, on the Grassmannian bundle is carried out. We then verify that these groups are zero. 
  
Another approach using vanishing theorems on Hecke curves is indicated in \S \ref{vanishingthmhecke}, to rule out rank one destabilizing subsheaves.
It is an interesting problem to see if the Hecke correspondence can be utilised to prove existence of a K\"ahler-Einstein metric on the moduli spaces.

{\Small
Acknowledgements: We thank J-M. Hwang for suggesting to look at the question in summer 2011, and for useful discussions. A part of this work was done at KIAS, Seoul (summer 2011) and the hospitality and support is gratefully acknowledged. We are grateful to P. Newstead for a careful reading of previous versions, pointing out errors and  making useful suggestions, especially to use Zariski trivialization (Lemma \ref{Zariskitrivial}), and help to simplify many of the arguments. Thanks also to C. Simpson for his helpful remarks, especially on the inclusion in Lemma \ref{sheafsurjective}. }


\section{Determinant of destabilizing subsheaves of the cotangent bundle}\label{stableeven}

We start with some preliminaries to fix notations and definitions we will use.

\subsection{Preliminaries}

Suppose $X$ is a projective manifold of dimension $n$ and $L$ is an ample line bundle on $X$.
Suppose $E$ is any coherent sheaf on $X$ of rank $k$ and degree d with respect to $L$. In other words, the determinant $\wedge^kE$ has the intersection number $d:=(c_1(\wedge^kE).c_1(L)^{n-1})$. The slope of $E$ is defined to be $\mu(E):=\frac{d}{k}$.  Mumford-Takemoto stability means that, for any coherent subsheaf $F\subset E$, $0<rank(F)< k$, we have the inequality:
$$
\mu(F)\,< \, \mu(E).
$$
If the above strict inequality $(<)$ is replaced by the inequality ($\leq$), then we say that $E$ is semistable.

In the proofs, we will need to look at sheaves on open smooth varieties but whose complementary locus in a compactification, has high codimension. We note the following lemma, which we will use.

\begin{lemma}\label{rem2}
Suppose $X$ is a projective variety and $U\subset X$ be an open smooth subset. Let $S:=X-U$ be the complementary closed subset and assume it has codimension at least two. Let $E$ be a coherent sheaf on $X$. Then $c_1(E)$ and $\mu(E)$ are well defined. In particular $c_1(E)$ and $\mu(E)$ are well-defined, when $X$ is a normal projective variety.
\end{lemma}
\begin{proof} 
See \cite[p.318-319]{Maruyama}. The key point is that $U=X-S$ is  smooth and the Chern class $c_1(E_{|U})$ of the restriction of $E$ on $U$ is well-defined, using a locally free resolution of $E_{|U}$. The Weil divisor $c_1(E_{|U})$ extends uniquely on $X$ since $codim(S)\geq 2$. Hence $c_1(E)$ and $\mu(E)$ are well-defined on $X$.
\end{proof}

\subsection{Determinant of destabilizing subsheaf of $\Omega_X$ on the moduli space $\cSU_C(r,d)$}


Suppose $C$ is a smooth projective curve of genus $g\geq 3$. Fix an integer $d$ and assume that $r$ is coprime to $d$.
Let $X:=\cSU_C(r,d)$ denote the moduli space of stable bundles of rank $r$ with fixed determinant $\eta$ of degree $d$ on $C$. Then $X$ is a projective manifold of dimension  $N:= (r^2-1)(g-1)$.

We note that the Picard number of $X$ is one and let $L$ be the ample generator of $\m{Pic}X$. Also $X$ is of index two, i.e.,
the canonical line bundle $K_X\,=\,L^{-2}$ (\cite[Theorem 1, p.69]{Ramanan}).

Since the dual of a stable bundle is again stable, it suffices to prove that the cotangent bundle $\Omega_X^1$ of the Fano manifold $X$ is stable.

We remark that the stability of the cotangent bundle is implied by the vanishing of some Hodge cohomologies twisted by appropriate powers of the ample class $L$. This can be seen as follows.
Suppose $S\subset \Omega_X^1$ is a coherent subsheaf of rank $s$ and $\wedge^s S\,=\,L^k$, for some integer $k$.
The inclusion of sheaves gives a non-trivial section of $\Omega_X^s\otimes L^{-k}$. The stability of the cotangent bundle will hold if we have the following vanishing:
$$
H^0(X,\Omega_X^s\otimes L^{-k})=0, \, \mbox{for } 0<\,s\,< N,\,\mbox{ and } k\geq s.\frac{-2}{N}.
$$
Since $K_X=L^{-2}$, the condition on the slope is
\begin{equation}\label{slopecondition}
\frac{k}{s} \geq \frac{-2}{N}, \mbox{ i.e. } -k \leq \frac{s}{N}.2 < 2.
\end{equation}

In this situation we note that the stability or semistability of $\Omega_X$ give the same inequality $-k\, <\,2$.

\begin{lemma}\label{slopek}
With notations as above, the only possibility for $k$ is equal to $-1$, i.e. $det\,S= L^{-1}$, for a destabilizing subsheaf $S\subset \Omega_X$.
\end{lemma}
\begin{proof}We exclude the other values of $k$ as follows:

\textbf{Case $-k < 0$} :
By Akizuki-Nakano vanishing theorem \cite{Ak-Na}, we have
$$
H^0(X,\Omega^s_X \otimes L^{-k}) \,=\,0 \,\mbox{ for any }0<\,s\,< N.
$$

\textbf{Case $k=0$} :
Since $X$ is Fano and hence rationally connected, the required Hodge cohomology vanish \cite[p.202]{Kollar}.

\textbf{Case $-k>0$}\label{descent} :
The slope condition \eqref{slopecondition} gives the only possibility $-k\,=\,1$.

\end{proof}

\subsection{Vanishing theorem on Hecke curves}\label{vanishingthmhecke}

In this subsection, we prove a vanishing theorem of sheaves on a Hecke curve inside the moduli space $M$ or $M'$. This will exclude the case of rank one subsheaves inside $T_M$ with determinant equal to $L$.

\begin{lemma}\label{ratcurve}
Suppose $R\subset M$ is a Hecke curve. Then the following vanishing holds on $R$:
$$
H^0(R,\Omega_M^h\otimes L^{-k})=0
$$
for $h>0$ and $k\,>\, 0$.
If $R\subset M$ is a very free rational curve, then the above vanishing holds for $h> 0$ and $k\geq 0$.
\end{lemma}
\begin{proof}
Let $R\subset M$ be a smooth Hecke curve passing through a general point of $M$.
Since a Hecke curve is free of minimal degree, the tangent bundle $T_M$ of $M$, restricted to $R$, splits as follows (\cite[p.195]{Kollar}):
\begin{equation}\label{ratrestrict}
{T_M}_{|R} \,=\, \cO(2) \oplus \cO(1)\oplus...\oplus \cO(1)\oplus \cO\oplus\cO\oplus...\oplus \cO.
\end{equation}
Here $\cO(1)$ occurs $(d-2)$ times in the above sum and $d:= (-K_M.R)$. 
Hence 
$$
{\Omega_M}_{|R}\,=\, \cO(-2) \oplus \cO(-1)\oplus...\oplus \cO(-1)\oplus \cO\oplus\cO\oplus...\oplus \cO.
$$
For $h>0$, we get the expression:
$$
{\Omega^h_M}_{|R}\,=\, \cO(-a_1) \oplus \cO(-a_2)\oplus...\oplus \cO(-a_u)\oplus \cO\oplus\cO\oplus...\oplus \cO,
$$
where $a_i$ are positive integers, for each $i$.

By \cite[Theorem 1, p.925]{Sun},
$d=2r$.
Since $K_M=L^{-2}$, the degree $(L.R)=r$. Hence the sheaf $\Omega^h_M\otimes L^{-k}$, for $k>0$, restricted to $R$ looks like:
\begin{equation}\label{ratample}
({\Omega^h_M}\otimes L^{-k})_{|R} = \cO(-a_1-kr) \oplus \cO(-a_2-kr)\oplus...\oplus \cO(-a_u-kr)\oplus \cO(-kr)\oplus\cO(-kr)\oplus...\oplus \cO(-kr).
\end{equation}
Since $R$ is a rational curve, this bundle on $R$ has no global sections. 

For the second assertion, since $M$ is rationally connected, it contains very free rational curves and which cover $M$. Note that if $R$ is a very free rational curve then in \eqref{ratrestrict}, there are no trivial factors $\cO_{R}$. Hence in 
\eqref{ratample}, all the factors are negative on $R$ even when $k=0$. Hence we again deduce the asserted vanishing. 
\end{proof}

\begin{corollary}\label{heckecor}
There is no subsheaf $\cF\subset T_M$ of rank one and $det(\cF)=L$. 
\end{corollary}
\begin{proof}
Suppose $R\subset M$ is a smooth Hecke curve. As in the proof of Lemma \ref{ratcurve}, we have the restriction:
$$
{T_M}_{|R} \,=\, \cO(2) \oplus \cO(1)\oplus...\oplus \cO(1)\oplus \cO\oplus\cO\oplus...\oplus \cO.
$$
Here $\cO(1)$ occurs $(d-2)$ times in the above sum and $d:= (-K_M.R)$. By \cite[Theorem 1, p.925]{Sun},
$d=2r$.
Since $K_M=L^{-2}$, the degree $(L.R)=r$. Now consider
$$
({T_M}\otimes L^{-1})_{|R} = \cO(2-r) \oplus \cO(1-r)\oplus...\oplus \cO(1-r)\oplus \cO(-r)\oplus\cO(-r)\oplus...\oplus \cO(-r).
$$
If $r\geq 3$, then this bundle on $R$ has no global sections. Since the Hecke curves cover $M$, we get
the vanishing $H^0(M, T_M\otimes L^{-1})=0$.  
This gives the assertion.
\end{proof}


\section{Hecke correspondence between moduli spaces}\label{Heckecorr}


Recall the Hecke correspondence introduced by Narasimhan and Ramanan in \cite{Narasimhan}. It was investigated further by Beauville, Laszlo and Sorger in \cite{Beauville} and by others. 

We refer to the Hecke correspondence and the properties we use from \cite[p.206]{Beauville}. Assume $r\geq 3$ in the rest of the paper and consider the correspondence:

\begin{eqnarray*}
 \cP &\sta{q'}{ \rar} & M'\\
   \downarrow q && \\
   M &&
   \end{eqnarray*}

Here $M:= \cSU_C(r,\eta)$ and $M':= \cSU_C(r,\eta')$, such that $\m{deg}\eta=1$  and $\m{deg} \eta'= 1-h$, for $0<h<r$. Hence  $1-r<1-h<1$. We can choose $h$ so that $1-h$ and $r$ are coprime. 
So the degree $d$ in statement of main Theorem \ref{MainTheorem} is $d:=1, 1-h$.

Hence $M$ and $M'$ are smooth projective moduli spaces of dimension $N:=(r^2-1)(g-1)$. Moreover they are Fano varieties of index $2$.
The family $\cP\rar M$ is a family of Grassmannian varieties $G(h,r)$ and $q':\cP\rar M'$ is a rational map. The general fibre of $q'$ is a Grassmannian $G(r-h,r)$. See  \cite[Proof of Lemma 10.3]{Beauville}.
Let $N':= \m{dim } \cP$.

As in the previous section, denote the ample generator of $\m{Pic}M$ by $L$ and of $\m{Pic}M'$ by $L'$.
Then the canonical class $K_\cP$ of $\cP$ is (\cite[Lemma 10.3, p.207]{Beauville}):
\begin{equation}\label{canclassP}
K_\cP= q^*L^{-1} \otimes q'^*L'^{-1}.
\end{equation}

\begin{lemma}\label{choiceU}
There is a subset $\cU\subset \cP$ such that $codim(\cP-\cU)\geq 2$.
Furthermore, $\cU$ contains the generic fibre of $q$ and $q'$, $codim(M-q(\cU))\geq 2$
 and $codim(M'-q'(\cU))\geq 2$.
\end{lemma} 
\begin{proof}
The map $q'$ is a morphism outside a codimension two subset of $\cP$, call this subset $\cU\subset \cP$ on which $q'$ is defined. Furthermore,  $q'$ is defined on a generic fibre 
of $q$ (see proof of \cite[Lemma 10.3]{Beauville}). In other words, the subset $\cU$ contains the generic fibre of $q$. We now show that it contains the generic fibre of $q'$, 
and $codim(\cP-\cU)\geq 2$.
  
We note that $Z:=M'-\m{image}(q')$ is a subset whose closure in $M'$, is of codimension at least two. 
Otherwise the closure $\bar{Z}$ is an effective divisor linearly equivalent to a positive multiple of $L'$ on $M'$. But $L'$ restricts on a generic fibre 
$G$ of $q$, as $\cO_G(r)$. Hence $\bar{Z}$ is of codimension at least two. Denote $\cU":= q'^{-1}(M'-\bar{Z})$. Then $\cU"\subset \cU$. Since $codim(\cP-\cU")\geq 2$, 
we deduce that $codim(\cP-\cU)\geq 2$ and $\cU$ contains a generic fibre of $q'$.
Similar argument holds for the morphism $q$ to deduce the last assertion on codimension of $M-q(\cU)$.
\end{proof}

 In the rest of the paper, we will consider sheaves on $\cU\subset \cP$, and in particular we will use that $\m{Pic}(q'(\cU))=\m{Pic}(M')=\Z.L'$ and $codim (\cP-\cU)\geq 2$.

We also note the following, on the structure of the Grassmannian bundles, with notations as in the previous lemma.

\begin{lemma}\label{Zariskitrivial}
The Grassmannian bundles $q:\cP\rar M$ and $q':\cU"\rar (M'-\bar{Z})\subset M'$ are Zariski locally trivial. Furthermore, codim $(\cP-\cU")\geq 2$.
\end{lemma}
\begin{proof}
Since the moduli space $M$ parametrises stable bundles of rank coprime to their degree, there is a Poincar\'e bundle on $C\times M$. The Grassmannian bundle $\cP\rar M$ is associated to (restriction of) the Poincar\'e bundle, see \cite[p.206]{Beauville}. 
Hence the fibration $q$ is Zariski locally trivial. Similar statement holds for $q'$ on $\cU"$, and the proof of above Lemma \ref{choiceU} shows that $codim(\cP-\cU")\geq 2$.
\end{proof}

In the later sections, we will use the following identification of the cohomology groups.

\begin{lemma}\label{proformula}
Suppose $\cE$ is a coherent sheaf on $\cP$. Then
$$
H^0(\cU,\cE) \,=\,H^0(M,q_*\cE).
$$
In particular when $\cE=q^*\cH$ for some coherent sheaf $\cH$ on $M$, we have the equality:
$$
H^0(\cU,\cE) \,=\,H^0(M,\cH).
$$
Here $\cU$ can be replaced by any open subset of $\cP$, of codimension at least two. This also holds for the map $q'$, where it is proper.
\end{lemma}
\begin{proof}
Since codimension of $\cP-\cU$ is at least two, by Hartog's theorem, we have the equality:
$$
H^0(\cU,\cE)\,=\, H^0(\cP,\cE).
$$
This gives us,
$$
H^0(\cP,\cE)= H^0(M,q_*\cE).
$$
Since $q$ is a proper morphism with connected fibres, 
when $\cE=q^*\cH$, by projection formula, we have $q_*q^*\cH\,=\, \cH$.
This gives the claim.

\end{proof}




\section{Stability of tangent bundle  under Hecke correspondence: $r\geq 3$}


In this subsection, we prove that the stability of the tangent bundle is preserved under the Hecke correspondence. This will help us to conclude the stability of the tangent bundle for any $\cSU_C(r,d)$, when $(r,d)=1$, in the final section.

Recall the Hecke correspondence from the previous section, and the choice of $\cU\subset \cP$,  such that $q':\cU\rar M'$ is a morphism, containing a generic fibre of $q$ and $q'$,  and $codim(\cP-\cU)\geq 2$.

Recall from \eqref{slopecondition} that stability of $\Omega_X$ or $T_X$ (here $(X,\cL):= (M,L) \m{ or } (M',L')$) will hold if
there does not exist a coherent subsheaf
$$
\cF\subset T_X
$$
of rank $p\leq \f{N}{2}$ and $det(\cF)=L$.

Consider the exact sequence of tangent sheaves on $\cU$ :
\begin{equation}\label{tanbundle}
0\rar T_{\cP/M}\rar T_\cP \rar q^*T_M\rar 0.
\end{equation}
A similar exact sequence corresponds to the fibration $q'$.

We now look at certain sheaf sequences to relate $T_{\cP/M}$ and $T_{\cP/M'}$ inside $T_\cP$.

Consider the product variety $M\times M'$ with the projections $l:M\times M'\rar M$, $l':M\times M'\rar M'$.
The inclusion in the following lemma was pointed out by C. Simpson.

\begin{lemma}\label{sheafsurjective}
There is a map:
$$
\gamma: \cP\rar M\times M'
$$
defined on $\cU$, and compatible with the projections $l,l',q,q'$. The map $\gamma$ is injective on $\cU$.
The natural  sheaf map:
$$
\gamma^*(\Omega_{M\times M'}=l^*\Omega_M \oplus l'^*\Omega_{M'} )\rar \Omega_\cP
$$
 is generically surjective. 
\end{lemma}
\begin{proof}
The existence of the map $\gamma$ and compatibility with projections, follows from universality of products over $Spec(\C)$.
More concretely, on $\cU$, $\gamma$ is given as $u \mapsto (q(u),q'(u))$.

To show that the natural sheaf map $\gamma^*(l^*\Omega_M\oplus l'^*\Omega_{M'})\rar \Omega_\cP$ is generically surjective, it suffices to show that the dual map on tangent spaces :
$$
T_\cP \rar T_{M}\oplus T_{M'},
$$
is injective.

This map restricted to $G$ is injective, since $q':\cU\rar M'$ restricted to $\{u\}\times G$ is injective, for any $u\in M $, whenever $q'$ is defined (see proof of \cite[Lemma 10.3]{Beauville}). In particular, 
$$
\gamma:\cU\rar M\times q'(\cU)
$$ 
is injective. 
\end{proof}

A dimension count shows that the image of $\cU$ under the map $\gamma$ has dimension strictly smaller than the product $M\times M'$.
We also note the following consequence of the inclusion $\cU\hookrightarrow M\times M'$.

\begin{lemma}\label{sheafiden}
There is a (generically) surjective map of sheaves on $\cU$:
$$
\Omega_\cU \rar \Omega_{\cP/M'}\oplus \Omega_{\cP/M}.
$$ 
In particular, there is a (generically) surjective map of sheaves, compatible with direct sums:
$$
q^*\Omega_M\oplus q'^*\Omega_{M'}\rar \Omega_{\cP/M'}\oplus \Omega_{\cP/M}.
$$
\end{lemma}
\begin{proof}
Using Lemma \ref{sheafsurjective}, we have the following inclusion of sheaves on $\cU$:
$$
 (q^*T_M \cap T_\cP) \oplus (q'^*T_{M'} \cap T_\cP)\hookrightarrow T_\cP \hookrightarrow q^*T_M \oplus q'^*T_{M'}. 
$$
Consider the exact sequence of tangent sheaves on $\cU$ associated to $\cP\rar M$ (and a similar one associated to $\cP\rar M'$):
$$
0\rar T_{\cP/M} \rar T_\cP \rar q^*T_M \rar 0.
$$
Now consider the exact sequence:
\begin{eqnarray}\label{splitseq}
 0\rar q'^*T_{M'} & \rar q^*T_M \oplus q'^*T_{M'} \rar & q^*T_{M} \rar 0\\
  \cup         & \cup  & = \\
0\rar q'^*T_{M'} \cap T_\cP & \rar\,\, T_\cP\,\, \rar & q^*T_{M}\rar 0. 
\end{eqnarray}

We deduce that (inside $q^*T_M \oplus q'^*T_{M'}$):
\begin{equation}\label{cap1}
 q'^*T_{M'} \cap T_\cP =T_{\cP/M}
\end{equation}
and similarly
\begin{equation}\label{cap2}
q^*T_M \cap T_\cP =T_{\cP/M'}.
\end{equation}
This gives the inclusion of sheaves on $\cU$:
\begin{equation}\label{suminclusion}
T_{\cP/M'}\oplus T_{\cP/M} \hookrightarrow T_\cP. 
\end{equation}
Since the sheaves are locally free, dualizing we get a (generically) surjective map of sheaves on $\cU$:
$$
\Omega_\cP \rar \Omega_{\cP/M'}\oplus \Omega_{\cP/M}.
$$
The second assertion in the lemma follows from the above arguments.
\end{proof}

Consider the exact sequence of tangent sheaves on $\cU$:
$$
0\rar T_{\cP/M'} \rar T_\cP \sta{\eta}{\rar} q'^*T_{M'}\rar 0.
$$
We now note the following lemma on the inverse image of a subsheaf in $q'^*T_{M'}$.

\begin{lemma}\label{sum+inclusion}
Suppose $\cF'\subset T_{M'}$ is a coherent subsheaf on $M'$. Assume that $T_{\cP/M}\subset q'^*\cF'$. Then the inverse image $\eta^{-1}(q'^*\cF')\subset T_\cP$ contains the relative tangent sheaves $T_{\cP/M}$ and $T_{\cP/M'}$.
\end{lemma}
\begin{proof}
Clearly the sheaf $\eta^{-1}(q'^*\cF')$ contains the sheaf $T_{\cP/M'}$.
We need to show the inclusion 
$$
T_{\cP/M}\subset \eta^{-1}(q'^*\cF').
$$
We use the two exact sequences in \eqref{splitseq} (replacing $M$ by $M'$):
\begin{eqnarray*}
 0\rar q^*T_{M} & \rar q^*T_M \oplus q'^*T_{M'} \sta{\delta}{\rar} & q'^*T_{M'} \rar 0\\
  \cup         & \cup\,d  & = \\
0\rar q^*T_{M} \cap T_\cP & \rar\,\, T_\cP\,\, \rar & q^*T_{M'}\rar 0. 
\end{eqnarray*}
Here $d$ is the inclusion map and $\delta$ is the projection to the second factor.
Then we can write $\eta= \delta \circ d$. Hence we have:
\begin{eqnarray*}
\eta^{-1}(q'^*\cF')&=& d^{-1}(\delta^{-1}(q'^*\cF'))\\
           &=&(q^*T_M\oplus q'^*\cF') \cap T_\cP,         
\end{eqnarray*}
as a subsheaf of $T_{\cP}\subset q^*T_M\oplus q'^*T_{M'}$.

This means that we have the inclusion:
\begin{equation}\label{suminclusion}
(q^*T_M\cap T_\cP) \oplus (q'^*\cF' \cap T_\cP) \,\subset\, \eta^{-1}(q'^*\cF').
\end{equation}
Using \eqref{cap1} and \eqref{cap2}, the first direct summand is $T_{\cP/M'}$, and the second direct summand satisfies
 $$
 q'^*\cF' \cap T_\cP\subset   q'^*T_{M'}\cap T_\cP = T_{\cP/M}.
 $$ 
 Since we assume that $T_{\cP/M}\subset q'^*\cF'$, we deduce the equality $q'^*\cF' \cap T_\cP=T_{\cP/M}$.
Hence, \eqref{suminclusion} implies that both the relative tangent sheaves are contained in $\eta^{-1}(q'^*\cF')$.

\end{proof}   

\begin{lemma}\label{crosssection}
There is a subscheme $W"\subset \cP$ and an open subscheme $U"\subset M$ such that $q:W"\rar U"$ is birational, and $W"$ contains points of codimension one (of its closure). If $Z\subset \cP$ is a closed subset of codimension at least two, then we also have $W"\cap Z=\emptyset$. Furthermore, there is a Grassmannian $G'$, which is a fibre of $q'$, such that  $G'_{W"}:=G'\cap W" \subset G'$ is a subset containing points of codimension one.
\end{lemma}
\begin{proof}
Since $q:\cP\rar M$ is a Zariski locally trivial fibration (see Lemma \ref{Zariskitrivial}), there is a section $s:M\rar \cP$ defined over an open subset $U\subset M$. We assume that the image of the section $s$ does not intersect the generic point of $Z$. Hence the section extends outside a codimension two subset of $M$. This means that there is a closed subvariety $W\subset \cP$ together with a finite morphism $W\rar M$ and which is birational on an open subscheme $W'\subset W$ and containing points of codimension one. In other words, there is an open subset $U"\subset M$ such that $codim (M-U")\geq 2$ and $W":=(W'-Z) \rar U"$ is birational. (This argument is due to L. Morel-Bailly).

Now we make a choice of $s$ to prove the second assertion.
We use the Zariski trivialization $q:U\times G \rar U$ as above. We also know that
$q: G' \rar M$ is an embedding, where $G'\subset \cU$ is a generic fibre of $q'$ (use Lemma \ref{choiceU}).
Denote $G'_o:= U\cap q(G')\subset U$. Then we have $q^{-1}(G'_o)= G'_o \times G\subset (U\times G)$. 
Now choose a section $s:U\rar U\times G$, such that $s(G'_o) = G'_o \times z$, for a $z\in G$. Now we extend $s$ over codimension one points by above procedure, to get a birational morphism $q: W"\rar U"$. It restricts to a birational morphism $G'_{W"} \rar G'_{U"} \subset U"$. Here $G'_{W"}\subset W"$ is the image of $s:G'_o \rar W"$ extended over a subscheme $G'_{U"}\subset U"$.  Since $W"$ contains codimension one points, the image $s(G'_{U"})=G'_{W"}\subset G'$ also contains codimension one points of $G'$.
\end{proof}

\begin{proposition}\label{heckestable}
Suppose $\cF'\subset T_{M'}$ is a coherent subsheaf of rank $p$ and $det(\cF')=L'$. Then it corresponds (under Hecke) to a coherent subsheaf $\tilde{\cF}\subset T_M$ of rank $p$ and $det(\tilde{\cF})=L$. In other words, a destabilizing subsheaf of $T_{M'}$ corresponds to a destabilizing subsheaf of $T_M$.
Similar statement holds, replacing $M$ by $M'$, $M'$ by $M$ and $L'$ by $L$.
\end{proposition}
\begin{proof}
Suppose $\cF'\subset T_{M'}$ is a coherent subsheaf of rank $p$ and $det(\cF')=L'$.
 
Consider the exact sequence of tangent sheaves on $\cU$:
$$
0\rar T_{\cP/M'} \rar T_\cP \sta{\eta}{\rar} q'^*T_{M'}\rar 0.
$$
Using \cite[proof of Lemma 10.3]{Beauville},  we notice that there is an inclusion of the sheaf $T_{\cP/M}\hookrightarrow q'^*T_{M'}$ on $\cU$.
It is well-known that the tangent bundle of the generic fibre $G$ of $q$ (which is a Grassmannian) is stable, with respect to $L'_{|G}=K_G^{-1}$.
Hence we assume that $\cF'\subset T_{M'}$ is a maximal destabilizing torsion free subsheaf and we have an inclusion 
\begin{equation}\label{Gincl}
T_{\cP/M}\subset q'^*\cF'.
\end{equation}

Consider the inverse sheaf $\eta^{-1}(q'^*\cF') \subset T_\cP$. Then this sheaf fits in an exact sequence:
$$
0\rar T_{\cP/M'}\rar \eta^{-1}(q'^*\cF') \sta{\eta}{\rar} q'^*\cF' \rar 0.
$$
Hence  $\eta^{-1}(q'^*\cF')$ is of rank $p+m$ (here $m:=dim G(h,r)$) and, by \eqref{Gincl} and Lemma \ref{sum+inclusion}, it contains both the sheaves $T_{\cP/M'}$ and $T_{\cP/M}$.

Consider the tangent exact sequence associated to $\cP\rar M$:
$$
0\rar T_{\cP/M}\rar T_{\cP} \sta{\beta}{\rar} q^*T_M \rar 0
$$
and the exact sequence associated to the projection $\eta^{-1}(q'^*\cF') \rar   q^*T_M$:
\begin{equation}\label{torfree}
0\rar T_{\cP/M}\rar \eta^{-1}(q'^*\cF') \rar \cF \rar 0.
\end{equation}
Here $\cF:=\beta(\eta^{-1}(q'^*\cF')) \subset q^*T_M$.

Hence rank $(\cF)=p$. We note the determinants of the above sheaves.

We have $det(T_{\cP/M'})=det(T_P)\otimes det(q'^*T_{M'})^{-1}= q^*L\otimes q'^*L'^{-1}$. Similarly $det(T_{\cP/M})=q'^*L'\otimes q^*L^{-1}$. Hence $det(\eta^{-1}(q'^*\cF'))=q^*L$ on $\cU$.
Since the sheaves in \eqref{torfree} are torsion free, they are locally free outside a codimension two subset $Z\subset \cU$.

By Lemma \ref{crosssection}, we can choose a section $s:U"\rar \cU$ such that $codim(M-U")\geq 2$ and there is a subvariety $W"\subset \cU$ such that the restriction of $q$ to $W"\rar U"$ is birational. 

We now pullback the exact sequence \eqref{torfree} on $U"$, via $s$, to get a short exact sequence of locally free sheaves on $U"$ :
\begin{equation*}
0\rar s^*T_{\cP/M} \rar s^*\eta^{-1}(q'^*\cF') \rar s^*\cF \rar 0.
\end{equation*}

Since $\m{Pic}(M)=\m{Pic}(U")=\Z.L$,  $det(s^*T_{\cP/M})=s^*(q'^*L'\otimes q^*L^{-1})=L^a$, for some $a\in \Z$.

We claim that $a=0$.

Suppose $a>0$.  In other words, $s^*(q'^*L'\otimes q^*L^{-1})$ is ample on $U"$. Next, using the second assertion of Lemma \ref{crosssection}, there is a fibre $G'$ of $q'$, such that $G'_{W"}\subset W"$, and $G'_{W"}\subset G'$ is a subscheme containing points of codimension one. 

 We note that $q'^*L'\otimes q^*L^{-1}$ restricted on the fibre $G'$ of $q'$ is $q^*L^{-1}$, which is not ample. Since the map $s$ is birational, the restriction
satisfies:
\begin{eqnarray*}
s^*(q'^*L'\otimes q^*L^{-1})_{| G'_{W"}} & = & s^*((q'^*L'\otimes q^*L^{-1})_{| G'_{W"}})\\
                                    & = & s^*(q^*L^{-1}) \\
                                    & = & L^{-1}.
\end{eqnarray*}
                                       
This means that the pullback $s^*(q'^*L'\otimes q^*L^{-1} )=L^a$ is not ample on $U"$. This is a contradiction to $a>0$. Hence $a\leq 0$.

Suppose $a<0$.  
Then the determinants of the sheaves in \eqref{torfree} pulled back on $U"$, via $s$, satisfy:
$$
s^*det(\eta^{-1}(q'^*\cF'))=s^*det (\cF)\otimes L^{-a}.
$$
  The  coherent subsheaf $s^*\cF\subset  T_{U"}$ is of rank $p$, on $U"$. Take a coherent extension $\tilde{\cF}\subset T_M$ on $M$ (see \cite[Ex.5.19 d), Chap.II]{Hartshorne}). Since codim$(M-U")\geq 2$ and $det(\eta^{-1}(q'^*\cF'))=q^*L$,   we deduce that $det(\tilde{\cF})=L^{1-a}$ on $M$, where $a<0$. Hence, it induces a nonzero section in
$H^0(M, \bigwedge^p T_M\otimes L^{a-1})= H^0(M,\Omega^{N-p}_M\otimes L^{1+a})$, since $K_M=L^{-2}$. However, this group is zero when $1+a< 0$, using Kodaira-Akizuki-Nakano theorem and when $1+a=0$, using rational connectedness of $M$.

Hence we conclude that $a=0$. In other words, we deduce that $det(\tilde{\cF})=L$.  
Since $p\leq \f{N}{2}$, $\tilde{\cF}$ is a destabilizing subsheaf of $T_M$ (see Lemma \ref{rem2}).

Above arguments also hold replacing $M$ by $M'$ and $M'$ by $M$, since we only need to note that $q'$ is defined on a generic fibre of $q$ (see proof of \cite[Lemma 10.3]{Beauville}).  Since $Pic(M')=Pic(q'(\cU))$, we can conclude by similar arguments as above.
\end{proof}



\section{Stability of the tangent bundle $T_M$: Main Theorem}
Recall the Hecke correspondence and notations, from \S \ref{Heckecorr}:

\begin{eqnarray*}
 \cP &\sta{q'}{ \rar} & M'\\
   \downarrow q && \\
   M. &&
   \end{eqnarray*}
Here $q'$ is a rational map and there is a subset $\cU\subset \cP$ where $q'$ is defined and such that $codim(\cP-\cU)\geq 2$. 
(The choice of $\cU$ is made in  Lemma \ref{choiceU}, with $codim(\cP-\cU)\geq 2$ and $\cU$ contains a generic fibre of $q$ and $q'$).

 In the rest of the  proofs, we will consider sheaves on $\cU\subset \cP$.

 
\subsection{Main theorem}

Now we proceed to show:

\begin{theorem}\label{finaltheorem}
The cotangent bundle $\Omega_M$ of the moduli space $M= \cSU_C(r,d)$ where $(r,d)=1$ and $r\geq 3$ is always stable.
\end{theorem}

We start by assuming to the contrary, and gather the consequences in this subsection. The next subsection will analyse the possible cases which will be ruled out, using rational connectedness and Kodaira-Akizuki-Nakano theorem.

Using Proposition \ref{heckestable}, we assume that both $\Omega_M$ and $\Omega_{M'}$ are not stable.
   
Suppose $\cF\subset \Omega_M$ (resp. $\cG\subset \Omega_{M'}$) is a coherent subsheaf destabilizing $\Omega_M$ (resp. $\Omega_{M'}$). Then we have seen in \S \ref{descent} that we have the following only possibilities  
\begin{equation}\label{detL}
det(\cF)=L^{-1},\, det(\cG)=L'^{-1}
\end{equation}
and 
\begin{equation}\label{detrank}
\m{rank}(\cF)\,=\,p\geq \f{N}{2},\, \m{rank}(\cG)\,=\,p'\geq \f{N}{2}.
\end{equation}   

Consider the product variety $M\times M'$ with the projections $l:M\times M'\rar M$, $l':M\times M'\rar M'$, and compatible with $q,q'$ as in Lemma \ref{sheafsurjective}, together with the  injective map 
$$
\gamma:\cU\rar M\times q'(\cU).
$$

Consider the map of sheaves, on $M\times q'(\cU)$ (use Lemma \ref{sheafiden}):
$$
\eta: l^*\cF \oplus l'^*\cG \hookrightarrow l^*\Omega_{M} \oplus l'^*\Omega_{M'} \rar \gamma_*\Omega_\cU  \rar \gamma_*(\Omega_{\cP/M'}\oplus \Omega_{\cP/M}).
$$
The composed map is taking sums inside $\gamma_*(\Omega_{\cP/M'}\oplus \Omega_{\cP/M})$.

This gives a left exact sequence on $M\times q'(\cU)$ :
\begin{equation}\label{exeexact}
0\rar \tilde{\cK} \rar l^*\cF \oplus l'^*\cG  \rar \gamma_*(\Omega_{\cP/M'}\oplus \Omega_{\cP/M}).
\end{equation}

Here $\tilde\cK:= \m{ kernel }(\eta)$. Note that the image of $\eta$ is supported on $\gamma(\cU)$, so the image has rank zero, and  hence
\begin{equation}\label{rankKtilde}
rank(\tilde\cK)=rank(l^*\cF \oplus l'^*\cG).
\end{equation}

We note the following lemma, which we will use.
\begin{lemma}\label{inversef}
There is a commutative diagram on $M\times q'(\cU)\subset M\times M'$:
\begin{eqnarray*}
     0 & 0 &   \\
   \downarrow & \downarrow & \\ 
0\rar \cT &\rar \cB \rar &  \\
 \downarrow & \downarrow &  \\
 0\rar \tilde{\cK} & \rar l^*\cF\oplus l'^*\cG \rar &  \gamma_*(\Omega_{\cP/M'}\oplus \Omega_{\cP/M})    \\
 \downarrow f & \downarrow & \downarrow \\
 \rar \tilde{\cK}\otimes \gamma_*\cO_{\cU} & \rar (l^*\cF \oplus l'^*\cF)\otimes \gamma_*\cO_{\cU} \rar &  \gamma_*(\Omega_{\cP/M'}\oplus \Omega_{\cP/M})\otimes \gamma_*\cO_\cU  \\
  \downarrow & \downarrow & \downarrow \\
  0 & 0 & 0. \\
  \end{eqnarray*}     
The columns are exact and rows are left exact. Here $\cT:= \tilde{\cK}. \cI_{\gamma(\cU)}$, where $\cI_{\gamma(\cU)}$ denotes the ideal sheaf of the subset 
$\gamma(\cU)$ in $M\times q'(\cU)$.  In particular we have the equalities:
 $$
 rank(\cT)=rank(\tilde{\cK})= rank(l^*\cF\oplus l'^*\cG).
 $$
\end{lemma}
\begin{proof}
This commutative diagram follows from restriction of \eqref{exeexact} on the image $\gamma(\cU)$ in $M\times q'(\cU)$.
The rank equality is clear, using \eqref{rankKtilde}, since $\tilde{\cK}\otimes \gamma_*\cO_{\cU}$ is supported on $\gamma(\cU)$, whose dimension is strictly smaller than $M\times M'$.
\end{proof}


Since $q=l \circ \gamma$ and $q'= l'\circ \gamma$, pullback of the left exact sequence \eqref{exeexact}, on $\cU$, via $\gamma$ gives the  exact sequence of sheaves:
\begin{equation}\label{exeexactU}
0\rar T\rar \gamma^*\tilde{\cK} \sta{\alpha}{\rar} q^*\cF \oplus q'^*\cG \rar q^*\cF+q'^*\cG \rar 0.
\end{equation}
The right most term is just  projecting the direct sum into $\gamma^*\gamma_*(\Omega_{\cP/M'}\oplus \Omega_{\cP/M})$. The following map
$$ 
\gamma ^*\gamma_*(\Omega_{\cP/M'}\oplus \Omega_{\cP/M}) \rar \Omega_{\cP/M'}\oplus \Omega_{\cP/M},
$$ 
is generically injective on $\cU$. This is because if we restrict $\gamma$ on the open subset $\cU"\subset \cU\subset \cP$ (see proof of Lemma \ref{choiceU} for definition of $\cU"$ with codim$(\cP-\cU")\geq 2$) the above map of sheaves is an isomorphism on $\cU"$.

Hence we obtain a generically injective map on $\cU$ (and injective on $\cU"$): 
\begin{equation}\label{finalsubsheaf}
q^*\cF+ q'^*\cG \rar \Omega_{\cP/M'}\oplus \Omega_{\cP/M}.
\end{equation}
The ranks of these sheaves will  be crucial in the rank estimates in the final section.

Here $T:=ker(\alpha)$. 

\begin{lemma}\label{detT}
There is a subsheaf $\tilde{\cT}\subset \tilde{\cK}$ on $M\times q'(\cU)$, such that 
$$
det(\gamma^*(\tilde{\cT}))\,=\,det(T) 
$$ 
and 
$$ 
rank({\cT})\,=\,rank(\tilde{\cT})\,=\,rank (l^*\cF\oplus l'^*\cG).
$$
\end{lemma}
\begin{proof}
Pushforward of \eqref{exeexactU} on $M\times q'(\cU)$ gives the exact sequence:
$$
0 \rar \gamma_*(T)\rar \tilde{\cK} \otimes \gamma_*(\cO_\cU) \rar (l^*\cF \oplus l'^*\cG)\otimes \gamma_*(\cO_{\cU})\rar.
$$
In the commutative diagram in Lemma \ref{inversef}, we take the inverse image $\tilde{\cT}:=f^{-1}(\gamma_*T)$ of $\gamma_*T$ under the map $f$, which gives a short exact sequence, compatible with the maps in Lemma \ref{inversef}:
$$
0\rar {\cT}\rar  \tilde{\cT} \rar \gamma_*(T) \rar 0.
$$
Since $\gamma_*(T)$ is supported on $\gamma(\cU)$, we get the rank equality:
$$
rank(\tilde{\cT})=rank(\cT).
$$
Pullback via $\gamma$ on $\cU$ gives the exact sequence:
$$
\gamma^*{\cT}\rar \gamma^*(\tilde{\cT})\rar \gamma^*\gamma_*(T) \rar 0.
$$
However the left hand map is zero. Furthermore, restricting $\gamma$ on the open subset $\cU"\subset \cU$, (for $\cU"$ as in  Lemma \ref{choiceU}), we note that $\gamma^*\gamma_*T=T$ on $\cU"$.  
Hence we get $\gamma^*(\tilde{\cT})= T$ on $\cU"$. Since codim$(\cP-\cU")\geq 2$, we get the equality of determinants \cite[p.10]{Huybrechts}:
$$
det(\gamma^*(\tilde{\cT}))\,=\,det(T) 
$$ 
on $\cP$.
\end{proof}

Here $\cK:= image(\gamma^*\tilde{\cK}\rar q^*\cF\oplus q'^*\cG)$. Then  we get the short exact sequences on $\cU$ (by breaking up \eqref{exeexactU}):
\begin{equation}\label{shortT}
0\rar T \rar \gamma^*\tilde{\cK} \rar \cK\rar 0
\end{equation}
and
\begin{equation}\label{shortK}
0\rar \cK \rar q^*\cF \oplus q'^*\cG \rar q^*\cF+q'^*\cG \rar 0.
\end{equation}

\begin{remark}
 We note that it is relevant to look at the exact sequence \eqref{exeexact} on the product $M\times q'(\cU)$ and then restrict on $\cU$, instead of taking sums on $\cU$ as in \eqref{shortK}.
This is essential to deduce the triviality of $det(\cK)$, in Corollary \ref{trivK}.
\end{remark}

In the following lemma we note the determinants, (see \cite{Mumford}, for a definition and functorial properties of the functor $det$).

\begin{lemma}\label{detTK}
We have
$$
a) \,\,det(\cT)=det(\tilde\cK),
$$
$$
b) \,\,det(\cT)=det(\tilde{\cT}),
$$
$$
c) \,\,det(T)\otimes det(\cK)= q^*L^{-1}\otimes q'^*L^{-1}
$$
and 
$$
d)\,\,\gamma^*det(\tilde\cT)=det(T).
$$
\end{lemma}
\begin{proof}
Using \eqref{exeexact}, we have on $M\times q'(U)$:
$$
det(\tilde{\cK})\otimes det(l^*\cF+l'^*\cG) = det (l^*\cF\oplus l'^*\cG)= l^*L^{-1}\otimes l'^*L'^{-1}.
$$
Since $Image(\eta)$, $\gamma_*(T)$ and $\tilde{\cK}\otimes \gamma_*(\cO_\cU)$ are supported on $\gamma(\cU)$ whose codimension is at least two, their determinants are trivial, i.e equal to $\cO_{M\times q'(\cU)}$ \cite[p.10]{Huybrechts}.
Similarly, we note that 
$$
det(\cT)=det(\tilde{\cK})
$$
and
$$
det(\cT)=det(\tilde{\cT}).
$$
Putting them together, using \eqref{shortT} and $det(\gamma^*\tilde{\cK})= \gamma^*(det(\tilde{\cK}))$, we get 
$$
det(T)\otimes det(\cK)= q^*L^{-1}\otimes q'^*L^{-1}.
$$
By Lemma \ref{detT},   we obtain $\gamma^*det(\tilde\cT)=det(T)$.
\end{proof}

Hence, we deduce:
\begin{corollary}\label{trivK}
$$
det(\cK)= \cO.
$$
\end{corollary}
\begin{proof}
Using \eqref{shortT} and, a) and d) of Lemma \ref{detTK}, we deduce that $det(\cK)=\cO$.
 
\end{proof}
\subsection{Proof of main theorem}

Recall the short exact sequence \eqref{shortK} on $\cU$:
$$
0\rar \cK \rar q^*\cF \oplus q'^*\cG \rar q^*\cF+q'^*\cG \rar 0.
$$

Let $s:= \m{rank}({\cK})$ and $s':=\m{rank}(q^*\cF + q'^*\cG)$ on $\cU$. (Note that these ranks may be different from the ranks on $M\times M'$).

Our aim is now to show that the short exact sequence \eqref{shortK} does not exist on $\cU$. To show this, we first  show that rank$(\cK)>0$. Then we consider the non-zero section induced by taking appropriate determinants. We then check that the section lies in a Hodge cohomology twisted by powers of $L$ and $L'$. Then we will note that these groups are zero, using rational connectedness and Kodair-Akizuki-Nakano theorem.
The relevant cohomology groups are detailed below, which depend on the powers of $L$ and $L'$. When $rank(\cK)=0$, we make a rank estimate of the direct sum with respect to the target sheaf, to obtain a contradiction.
 
 Before analysing the induced sections, we note the following lemma, which we will need.
 
 \begin{lemma}\label{incluK}
 Consider the subsheaf $\cK\subset q^*\cF \oplus q'^*\cG$, from \eqref{shortK}. Then there is a short exact sequence of torsion free sheaves on $\cU$: 
 $$
 0\rar K \rar \cK \rar K' \rar 0,
 $$
 such that the sheaves $K$ and $K'$ admit generically injective maps into $q^*\Omega_M\cap q'^*\Omega_{M'} \subset \Omega_\cU$.
\end{lemma}
\begin{proof} Since $q^*\cF$ and $q'^*\cG$ are torsion free sheaves, we note that the subsheaf $\cK\subset q^*\cF\oplus q'^*\cG$ is also torsion free. Now
we use the generically injective map \eqref{finalsubsheaf} on $\cU$:
$$
q^*\cF+ q'^*\cG \rar \Omega_{\cP/M'}\oplus \Omega_{\cP/M}.
$$
and realize it as a subsheaf of $\Omega_{\cP/M'}\oplus \Omega_{\cP/M}$ on $\cU"\subset \cU$. Here $\cU"\subset \cU$ is defined in  Lemma \ref{choiceU}, where the above map is actually injective and hence $q^*\cF+q'^*\cG$ is also torsion free on $\cU"$.

Denote $(q^*\cF+ q'^*\cG)_{\cU"}$ the sum of $q^*\cF$ and $ q'^*\cG$ taken inside $\Omega_{\cU"}$.

Consider the commutative diagram on $\cU"$:
\begin{eqnarray*}
& & 0 \,\, \,\,\,\,\, \,\,\,\,\,\,\,\,\,\,\,\,\,\,\,\,\,\,\,\,\,\,\,\,\,\,\,\,\,\,\,\,\,\,\,\,\,\,\,\,\,\,\,\,\,\,\,\,\,\,\,\,\,\,\,\,\,\,\,\,\,\,\,\,\,\,\,\,\,\,\,\,\,\,\,\,\,\,\,\,\,\,\,\,\,\,\,\,0 \\
   \,\,\,\,\,\,\,\,\,        &   &   \downarrow  \,\,\,\,\, \,\,\,\,\,\,\,\,\,\,\,\,\,\,\,\,\,\,\,\,\,\,\,\,\,\,\,\,\,\,\,\,\,\,\,\,\,\,\,\,\,\,\,\,\,\,\,\,\,\,\,\,\,\,\,\,\,\,\,\,\,\,\,\,\,\,\,\,\,\,\,\,\,\,\,\,\,\,\,\,\,\,\,\,\,\,\,\,\,\,\downarrow \\
0\rar K_{\cU"} \rar & \,\,\,\,\,\,\,\,\,\,\,\,\cK \,\,\,\,\,\,\,\,\,\rar & (q^*\cF + q'^*\cG)_{\cU"}\cap (q^*\Omega_M \cap q'^*\Omega_{M'}) \subset \,\, q^*\Omega_M \cap q'^*\Omega_{M'} \\
 ||  \,\,\,\,\,\,\,\,\,        & \downarrow &   \downarrow  \,\,\,\,\, \,\,\,\,\,\,\,\,\,\,\,\,\,\,\,\,\,\,\,\,\,\,\,\,\,\,\,\,\,\,\,\,\,\,\,\,\,\,\,\,\,\,\,\,\,\,\,\,\,\,\,\,\,\,\,\,\,\,\,\,\,\,\,\,\,\,\,\,\,\,\,\,\,\,\,\,\,\,\,\,\,\,\,\,\,\,\,\,\,\,\downarrow \\
0\rar K_{\cU"} \rar & q^* \cF \oplus q'^* \cG  \sta{t}{\rar} & (q^*\cF + q'^* \cG)_{\cU"} \,\,\,\,\,\,\,\subset \,\,\,\,\,\,\, \,\,\,\,\,\,\,\,\,\,\,\,\,\,\,\,\,\,\,\,\,\,\,\,\,\,\,\,\,\,\,\,\,\,\,\,\Omega_{\cU"} \\
\downarrow \,\,\,\,\,\,\,\,\,   & ||                       & \downarrow  \,\,\,\,\,\,\,\, \,\,\,\,\,\,\,\,\,\,\,\,\,\,\,\,\,\,\,\,\,\,\,\,\,\,\,\,\,\,\,\,\,\,\,\,\,\,\,\,\,\,\,\,\,\,\,\,\,\,\,\,\,\,\,\,\,\,\,\,\,\,\,\,\,\,\,\,\,\,\,\,\,\,\,\,\,\,\,\,\,\,\,\,\,\,\,\,\downarrow \\
0\rar \cK \rar & q^*\cF \oplus q'^*\cG \rar &  q^*\cF + q'^* \cG \,\,\,\,\,\,\,\,\,\,\,\,\,\,\,\,\subset \,\,\,\,\,\,\,\,\,\,\,\,\,\,\,\,\,\,\,\,\,\,\,\,\,\,\,\,\,\,\,\,\,\,\,\,\Omega_{\cP/M'} \oplus \Omega_{\cP/M} \\
 &  &  \downarrow \,\,\,\,\,\,\,\,\, \,\,\,\,\,\,\,\,\,\,\,\,\,\,\,\,\,\,\,\,\,\,\,\,\,\,\,\,\,\,\,\,\,\,\,\,\,\,\,\,\,\,\,\,\,\,\,\,\,\,\,\,\,\,\,\,\,\,\,\,\,\,\,\,\,\,\,\,\,\,\,\,\,\,\,\,\,\,\,\,\,\,\,\,\,\,\,\downarrow \\
 & &   0 \,\, \,\,\,\,\, \,\,\,\,\,\,\,\,\,\,\,\,\,\,\,\,\,\,\,\,\,\,\,\,\,\,\,\,\,\,\,\,\,\,\,\,\,\,\,\,\,\,\,\,\,\,\,\,\,\,\,\,\,\,\,\,\,\,\,\,\,\,\,\,\,\,\,\,\,\,\,\,\,\,\,\,\,\,\,\,\,\,\,\,\,\,\,\,\,\,0. \\ 
 \end{eqnarray*}
 The two columns on the right are exact sequences, and the bottom row is the same as \eqref{shortK}.
 The sheaf $K_{\cU"}$ is the kernel of the morphism $t$, on $\cU"$.
  
 Since the sheaf $K_{\cU"}$ is generically $q^*\cF\cap q'^*\cG$ (inside $q^*\Omega_M\cap q'^*\Omega_{M'}$), 
 there is a generically injective map
  $$
 K_{\cU"}\hookrightarrow  q^*\Omega_M\cap q'^*\Omega_{M'}.
 $$
Denote
$$
K'_{\cU"}:= (q^*\cF + q'^*\cG)_{\cU"}\cap (q^*\Omega_M \cap q'^*\Omega_{M'}) \,\subset \, q^*\Omega_M \cap q'^*\Omega_{M'}
$$ 
on $\cU"$.
However, since $\cK$ is a torsion free sheaf on $\cU$,  $K_{\cU"}$ extends as a torsion free subsheaf $K$, of $\cK$ on $\cU$.
This gives the exact sequence
$$
0\rar K \rar \cK \rar K'\rar 0
$$
as claimed, where $K'$ restricts to $K'_{\cU"}$ on $\cU"$. We can assume that $K'$ is also torsion free (after replacing $K'$ by $\f{K'}{{torsion}}$, and taking $K$ to be the kernel of the projection $\cK\rar \f{K'}{{torsion}}$. The generic injectivity statement still holds since on $\cU"$ there is no torsion).
 \end{proof}

We first note the following.
\begin{lemma}
The rank of the sheaf $\cK$ is non-zero.
\end{lemma} 
\begin{proof}
Suppose $rank (\cK)=0$.

In this case, we use the fact that $\cF$ and $\cG$ are torsion free sheaves and hence the sheaf $\cK=0$ on $\cU$.
Hence we have an inclusion of sheaves:
$$
q^*\cF \oplus q'^*\cG \hookrightarrow \Omega_{\cP/M'}\oplus \Omega_{\cP/M}.
$$
Now we check that this is not possible by computing the ranks
$$
(r^2-1)(g-1)=:N\leq rank(q^*\cF \oplus q'^*\cG ) \leq 2.dim(G) \leq r^2.
$$
This inequality does not hold if $r\geq 3,\,g\geq 3$ and we get a contradiction.

\end{proof}

 We now proceed as follows:

Using above lemma, we have $rank ({\cK})\neq 0$, and by Corollary \ref{trivK}, we know $det(\cK)=\cO$.

To compute the determinants of above sheaves, we use the fact that the Picard group of $\cP$ is generated by  $q^*\m{Pic}M$ and $\cO_{\cP}(1)$.
Note that this also holds outside a codimension two subset of $\cP$, in particular on $\cU$.

We can write 
$$
\m{det}(\cK)\,=\, \cO.
$$
  The above exact sequence \eqref{shortK}, and \eqref{detL} give
$$
\m{det}(q^*\cF + q'^*\cG)\,=\, q^*L^{-1}\otimes q'^*L'^{-1}.
$$  

Using the exact sequence in Lemma \ref{incluK}, we can write on $\cU$:
$$
det(\cK)=det(K).det(K').
$$
Let $det(K')\,=\,L^{a_1}\otimes L'^{b_1}$, then $det(K)=L^{-a_1}\otimes L'^{-b_1}$.

Using Lemma \ref{incluK},  on $\cU$, there are nonzero composed maps (which are generically injective): 

\begin{equation}\label{incl0}
K\hookrightarrow q^*\Omega_M\cap q'^*\Omega_{M'} \subset q^*\Omega_M,\,\,\,K'\hookrightarrow q^*\Omega_M\cap q'^*\Omega_{M'} \subset q^*\Omega_M,
\end{equation}
and
\begin{equation}\label{incl1}
K\hookrightarrow  q^*\Omega_M\cap q'^*\Omega_{M'} \subset q^*\Omega_{M'},\,\,\,K'\hookrightarrow  q^*\Omega_M\cap q'^*\Omega_{M'} \subset q^*\Omega_{M'}.
\end{equation}

\textbf{Case I:} Suppose that $rank(K)>0$ and $rank(K')>0$.

Hence taking determinants in \eqref{incl0}, we get nonzero morphisms
$$
q^*L^{-a_1}\otimes q'^*L'^{-b_1} \rar \bigwedge^sq^*\Omega_M,\,\, q^*L^{a_1}\otimes q'^*L'^{b_1} \rar \bigwedge^{s'}q^*\Omega_{M}.
$$
Here $s:= rank(K)$ (resply $s':=rank(K')$).
This give nonzero sections  in
$$
H^0(\cU, \bigwedge^t q^*\Omega_M\otimes q^*L^{-a'}\otimes q'^*L'^{-b'})
$$
when $(t,a',b'):= (s,-a_1, -b_1)$ and when $(t,a',b'):=(s',a_1,b_1)$.

Restricting on a generic fibre $G$ of $q$, we get a nonzero section in
\begin{equation}\label{zerocoh}
H^0(G,q'^*L'^{-b'}).
\end{equation}

If $b'\,>\, 0$, then the cohomology $H^0(G,q'^*L'^{-b'})$ is zero, since $q'^*L'$ on $G$ is $\cO_G(r)$ (see proof of \cite[Lemma 10.3]{Beauville}).
Hence the group in  \eqref{zerocoh} is zero.

Hence $b'\,\leq\,0$.

 Similarly, using \eqref{incl1}, if we consider  the fibration $q': \cP\rar M'$ we deduce that $a'\,\leq\,0$.
 
Hence 
$$
a',\,b'\leq 0. 
$$

Substituting the values of $(a',b')$, we deduce that 
$$
-a_1\leq 0,\, -b_1\leq 0,\, a_1\leq 0,\, b_1\leq 0.
$$
This implies that $a_1=b_1= 0$.

\textbf{Case II}: Suppose $rank(K)=0$ or $rank(K')=0$. Then since $\cF$ and $\cG$ are torsion free sheaves on $\cU$, the sheaf $K=0$ (resp $K'=0$), since they are torsion free (see Lemma \ref{incluK}) and hence 
\begin{equation}\label{ab}
a_1\,=\,b_1\,=\, 0. 
\end{equation}
Since $rank(\cK)\neq 0$, either $K$ or $K'$ has non zero rank.

Hence we get a non zero section in $H^0(\cU,q^*\Omega^t_M)$, for some $t=s>0$ or $t=s'>0$. Since $\cP-\cU\subset \cP$ is of codimension at least two, $M-q(\cU)\subset M$ is of codimension at least two (see proof of Lemma \ref{choiceU}), by  Lemma \ref{proformula}, the non zero section extends to give a non-zero section in $H^0(M, \Omega^t_M) $. But this group is zero using rational connectedness of $M$, since $t>0$.

This completes the proof.




\end{document}